\documentclass[12pt]{article}
\usepackage{amsmath,amsthm, amssymb}
\usepackage{amssymb,latexsym}

\newtheorem{theorem}{Theorem}

\newtheorem{lemma}{Lemma}

\newcommand{\bp}{\begin{proof}[Proof:]}
\newcommand{\ep}{\end{proof}}

\begin{document}

\baselineskip=17pt

\title{\bf Exponential sums over Piatetski-Shapiro primes in arithmetic progressions}

\author{\bf S. I. Dimitrov}
\date{2022}
\maketitle
\begin{abstract}
Let $\gamma<1<c$ and $19(c-1)+171(1-\gamma)<9$.
In this paper, we establish an asymptotic formula for exponential sums over
Piatetski-Shapiro primes $p=[n^{1/\gamma}]$ in arithmetic progressions.
\quad\\
{\bf  2020 Math.\ Subject Classification}: 11L07 $\cdot$  11L20 $\cdot$ 11N05 $\cdot$ 11N25
\end{abstract}

\section{Notations}
\indent

Let $x$ be a sufficiently large positive number.
The letter $p$  with or without subscript will always denote prime number.
By $\delta$ we denote a fixed positive number, it can be chosen arbitrarily small;
its value needs not be the same in all occurrences.
As usual $\Lambda(n)$ denotes von Mangoldt's function.
We use $[t]$ and $\{t\}$ to denote the integer part, respectively, the fractional part of $t$.
Moreover $e(y)=e^{2\pi i y}$ and $\psi(t)=\{t\}-1/2$.
We denote by $\tau _k(n)$ the number of solutions of the equation $m_1m_2\ldots m_k$ $=n$ in natural numbers $m_1,\,\ldots,m_k$.
Instead of $m\equiv n\,\pmod {k}$ we write for simplicity $m\equiv n\,(k)$.
Throughout this paper unless something else is said, we suppose that $\gamma<1<c$ and $19(c-1)+171(1-\gamma)<9$.
Define
\begin{align}
\label{pixdatc}
&\pi(x,d,a,t,c)=\sum\limits_{p\leq x\atop{p\equiv a\,(d)}}e(t p^c)\,;\\
\label{pigammaxdatc}
&\pi_\gamma(x,d,a,t,c)=\sum\limits_{p\leq x\atop{p=[n^{1/\gamma}]\atop{p\equiv a\,(d)}}}e(t p^c)\,.
\end{align}

\section{Introduction and statement of the result}
\indent

Let $\mathbb{P}$ denotes the set of all prime numbers.
In 1953 Piatetski-Shapiro \cite{Shapiro} has shown that for any fixed $\frac{11}{12}<\gamma< 1$ the set
\begin{equation*}
\mathbb{P}_\gamma=\{p\in\mathbb{P}\;\;|\;\; p= [n^{1/\gamma}]\;\; \mbox{ for some } n\in \mathbb{N}\}
\end{equation*}
is infinite.
The prime numbers of the form $p = [n^{1/\gamma}]$ are called Piatetski-Shapiro primes of type $\gamma$.
Denote
\begin{equation*}
\pi_\gamma(x)=\sum\limits_{p\leq x\atop{p=[n^{1/\gamma}]}}1\,.
\end{equation*}
Piatetski-Shapiro's result states that
\begin{equation}\label{Shapiroformula}
\pi_\gamma(x)=\frac{x^\gamma}{\log x}+\mathcal{O}\left(\frac{x^\gamma}{\log^2x}\right)
\end{equation}
for $\frac{11}{12}<\gamma< 1$.
The admissible range for $\gamma$ in this theorem has been extended many times over the years,
and the best results up to now belong to Rivat and Sargos \cite{Rivat-Sargos} with \eqref{Shapiroformula}
for $\frac{2426}{2817}<\gamma<1$ and to Rivat and Wu \cite{Rivat-Wu} with
\begin{equation*}
\pi_\gamma(x)\gg \frac{x^\gamma}{\log x}
\end{equation*}
for $\frac{205}{243}<\gamma<1$.
On the other hand Siegel-Walfisz theorem is extremely important result in analytic number theory
and plays a significant role in various applications.
It is a refinement both of the prime number theorem and of Dirichlet's theorem on primes in arithmetic progressions.
It states that for any fixed  $A > 0$ there exists a positive constant $c$ depending only on $A$ such that
\begin{equation}\label{Siegel-Walfisz}
\sum_{n\le x\atop{n\equiv a\,(d)}} \, \Lambda(n)
=\frac{x}{\varphi(d)}+\mathcal{O}\bigg(\frac{x}{e^{c\sqrt{\log x}}}\bigg)\,,
\end{equation}
whenever $x\geq2$, $(a,d) = 1$, $d\leq (\log x)^A$ and  $\varphi (n)$ is Euler's function.

In 2013 Baker, Banks,  Br\"{u}dern,  Shparlinski and Weingartner \cite{Baker-Banks}
have considered for first time the distribution Piatetski-Shapiro primes in arithmetic progressions.
They showed that when $a$ and $d$ are coprime integers, then for any fixed $\frac{17}{18}<\gamma<1$ we have
\begin{equation*}
\sum\limits_{p\leq x\atop{p=[n^{1/\gamma}]\atop{p\equiv a\,(d)}}}1
=\gamma x^{\gamma-1}\sum\limits_{p\leq x\atop{p\equiv a\,(d)}}1
+\gamma(1-\gamma)\int\limits_2^x y^{\gamma-2}\sum\limits_{p\leq y\atop{p\equiv a\,(d)}}1\,dy
+\mathcal{O}\bigg(x^{\frac{17}{39}+\frac{7\gamma}{13}+\delta}\bigg)\,.
\end{equation*}
In 2015 Guo \cite{Guo1}  enlarged the range of $\gamma$ to $\frac{13}{14}<\gamma<1$.
Recently Guo, Li and Zhang \cite{Guo2} achieved the best result up to now with $\frac{11}{12}<\gamma<1$.
Motivated by these results we establish an asymptotic formula for exponential sums over Piatetski-Shapiro primes
in arithmetic progressions. More precisely we prove the following theorem.
\begin{theorem}\label{Theorem} Let $a$ and $d$ be coprime integers, $d\geq1$.
Assume that $|t|\leq x^\delta$ with a sufficiently small $\delta>0$ and $\gamma<1<c$ which satisfy
\begin{equation}\label{cgamma}
\frac{19}{9}(c-1)+19(1-\gamma)<1\,.
\end{equation}
Then the sums \eqref{pixdatc} and \eqref{pigammaxdatc} are connected by the asymptotic formula
\begin{align}\label{asymptoticformula1}
\pi_\gamma(x,d,a,t,c)=\gamma x^{\gamma-1}\pi(x,d,a,t,c)
&+\gamma(1-\gamma)\int\limits_2^x y^{\gamma-2}\pi(y,d,a,t,c)\,dy\nonumber\\
&+\mathcal{O}\bigg(x^{\frac{c}{18}+\frac{\gamma}{2}+\frac{143}{342}+\delta}\bigg)\,.
\end{align}
\end{theorem}

\section{Preliminary lemmas}
\indent

\begin{lemma}\label{Vaaler}
For every $H\geq1$, we have
\begin{equation*}
\psi(t)=\sum\limits_{1\leq|h|\leq H}a(h)e(ht)+\mathcal{O}\Bigg(\sum\limits_{|h|\leq H}b(h)e(ht)\Bigg)\,,
\end{equation*}
where
\begin{equation}\label{bh}
a(h)\ll\frac{1}{|h|}\,,\quad b(h)\ll\frac{1}{H}\,.
\end{equation}
\end{lemma}
\begin{proof}
See \cite{Vaaler}.
\end{proof}

\begin{lemma}\label{Korput}
Suppose that $f''(t)$ exists, is continuous on $[a,b]$ and satisfies
\begin{equation*}
f''(t)\asymp\lambda\quad(\lambda>0)\quad\mbox{for}\quad t\in[a,b]\,.
\end{equation*}
Then
\begin{equation*}
\bigg|\sum_{a<n\le b}e(f(n))\bigg|\ll(b-a)\lambda^\frac{1}{2}+\lambda^{-\frac{1}{2}}\,.
\end{equation*}
\end{lemma}
\begin{proof}
See (\cite{Titchmarsh}, Theorem 5.9).
\end{proof}

\begin{lemma}\label{Sargos}
Suppose that $f'''(t)$ exists, is continuous on $[a,b]$ and satisfies
\begin{equation*}
f'''(t)\asymp\lambda\quad(\lambda>0)\quad\mbox{for}\quad t\in[a,b]\,.
\end{equation*}
Then
\begin{equation*}
\bigg|\sum_{a<n\le b}e(f(n))\bigg|\ll(b-a)\lambda^\frac{1}{6}+\lambda^{-\frac{1}{3}}\,.
\end{equation*}
\end{lemma}
\begin{proof}
See (\cite{Sargos}, Corollary 4.2).
\end{proof}

\begin{lemma}\label{Squareoutlemma}
Let $I$ be a subinterval of $(X, 2X]$.   For any complex numbers $z(n)$ we have
\begin{equation*}
\bigg|\sum_{n\in I}z(n)\bigg|^2
\leq\bigg(1+\frac{X}{Q}\bigg)\sum_{|q|< Q}\bigg(1-\frac{|q|}{Q}\bigg)
\sum_{n,\, n+q\,\in I}z(n+q)\overline{z(n)}\,,
\end{equation*}
where $Q\geq1$.
\end{lemma}
\begin{proof}
See (\cite{Heath2}, Lemma 5).
\end{proof}

\begin{lemma}\label{Heath-Brown} Let $G(n)$ be a complex valued function.
Assume further that
\begin{align*}
&P>2\,,\quad P_1\le 2P\,,\quad  2\le U<V\le Z\le P\,,\\
&U^2\le Z\,,\quad 128UZ^2\le P_1\,,\quad 2^{18}P_1\le V^3\,.
\end{align*}
Then the sum
\begin{equation*}
\sum\limits_{P<n\le P_1}\Lambda(n)G(n)
\end{equation*}
can be decomposed into $O\Big(\log^6P\Big)$ sums, each of which is either of Type I
\begin{equation*}
\mathop{\sum\limits_{M<m\le M_1}a(m)\sum\limits_{L<l\le L_1}}_{P<ml\le P_1}G(ml)
\end{equation*}
and
\begin{equation*}
\mathop{\sum\limits_{M<m\le M_1}a(m)\sum\limits_{L<l\le L_1}}_{P<ml\le P_1}G(ml)\log l\,,
\end{equation*}
where
\begin{equation*}
L\ge Z\,,\quad M_1\le 2M\,,\quad L_1\le 2L\,,\quad a(m)\ll \tau _5(m)\log P
\end{equation*}
or of Type II
\begin{equation*}
\mathop{\sum\limits_{M<m\le M_1}a(m)\sum\limits_{L<l\le L_1}}_{P<ml\le P_1}b(l)G(ml)
\end{equation*}
where
\begin{equation*}
U\le L\le V\,,\quad M_1\le 2M\,,\quad L_1\le 2L\,,\quad
a(m)\ll \tau _5(m)\log P\,,\quad b(l)\ll \tau _5(l)\log P\,.
\end{equation*}
\end{lemma}
\begin{proof}
See (\cite{Heath1}).
\end{proof}

\begin{lemma}\label{Optimization}
Let
\begin{equation*}
L(H)=\sum\limits_{i=1}^{m}A_iH^{a_i}+\sum\limits_{j=1}^{n}B_jH^{-b_j}\,,
\end{equation*}
where $A_i$, $B_j$, $a_i$ and $b_j$ are positive. Assume further that $H_1\leq H_2$.
Then there exists $H_0\in [H_1, H_2]$ such that
\begin{equation*}
L(H_0)\ll\sum\limits_{i=1}^{m}A_iH_1^{a_i}+\sum\limits_{j=1}^{n}B_jH_2^{-b_j}
+\sum\limits_{i=1}^{m}\sum\limits_{j=1}^{n}\Big(A_i^{b_j}B_j^{a_i}\Big)^{\frac{1}{a_i+b_j}}\,.
\end{equation*}
\end{lemma}
\begin{proof}
See (\cite{Srinivasan}, Lemma 3).
\end{proof}

\section{Proof of the theorem}
\indent

From \eqref{pigammaxdatc} we have
\begin{equation}\label{Gamma12}
\pi_\gamma(x,d,a,t,c)=\sum\limits_{p\leq x\atop{p\equiv a\,( d)}}
\big([-p^\gamma]-[-(p+1)^\gamma]\big)e(t p^c)=\Gamma_1(x)+\Gamma_2(x)\,,
\end{equation}
where
\begin{align}
\label{Gamma1}
&\Gamma_1(x)=\sum\limits_{p\leq x\atop{p\equiv a\,( d)}}\big((p+1)^\gamma-p^\gamma\big)e(t p^c)\,,\\
\label{Gamma2}
&\Gamma_2(x)=\sum\limits_{p\leq x\atop{p\equiv a\,( d)}}\big(\psi(-(p+1)^\gamma)-\psi(-p^\gamma)\big)e(t p^c)\,.
\end{align}

\subsection{Asymptotic formula for $\mathbf{\Gamma_1(x)}$}
\indent

Using \eqref{Gamma1}, the well-known asymptotic formula
\begin{equation*}
(p+1)^\gamma-p^\gamma=\gamma p^{\gamma-1}+\mathcal{O}\left(p^{\gamma-2}\right)
\end{equation*}
and Abel's summation formula we write

\begin{align}\label{Gamma1est1}
\Gamma_1(x)&=\gamma\sum\limits_{p\leq x\atop{p\equiv a\,( d)}}p^{\gamma-1}e(t p^c)+\mathcal{O}(1)\nonumber\\
&=\gamma x^{\gamma-1}\pi(x,d,a,t,c)
+\gamma(1-\gamma)\int\limits_2^x\left(\sum\limits_{p\leq y\atop{p\equiv a\,( d)}}e(t p^c)\right)y^{\gamma-2}\,dy+\mathcal{O}(1)\nonumber\\
&=\gamma x^{\gamma-1}\pi(x,d,a,t,c)+\gamma(1-\gamma)\int\limits_2^x y^{\gamma-2}\pi(y,d,a,t,c)\,dy+\mathcal{O}(1)\,.
\end{align}

\subsection{Upper bound for $\mathbf{\Gamma_2(x)}$}
\indent

Define
\begin{align}
\label{Gamma3}
&\Gamma_3(x)=\sum\limits_{p\leq x\atop{p\equiv a\,( d)}}(\log p)\big(\psi(-(p+1)^\gamma)-\psi(-p^\gamma)\big)e(t p^c)\,,\\
\label{Gamma4}
&\Gamma_4(x)=\sum\limits_{n\leq x\atop{n\equiv a\,( d)}}\Lambda(n)\big(\psi(-(n+1)^\gamma)-\psi(-n^\gamma)\big)e(t n^c)\,.
\end{align}
On the one hand  \eqref{Gamma2}, \eqref{Gamma3} and Abel's summation formula yield
\begin{equation}\label{Gamma23}
\Gamma_2(x)=\frac{\Gamma_3(x)}{\log x}+\int\limits_2^x \frac{\Gamma_3(y)}{y\log^2y}\,dy\,.
\end{equation}
On the other hand \eqref{Gamma3} and \eqref{Gamma4} give us
\begin{equation}\label{Gamma34}
\Gamma_3(x)=\Gamma_4(x)+\mathcal{O}\big(\sqrt{x}\big)\,.
\end{equation}
Splitting the range of $n$  into dyadic subintervals of the form $(x/2, x]$ from \eqref{Gamma4} we get
\begin{equation}\label{Gamma45}
\Gamma_4(x)\ll(\log x)\big|\Gamma_5(x)\big|\,,
\end{equation}
where
\begin{equation}\label{Gamma5}
\Gamma_5(x)=\sum\limits_{x/2<n\leq x\atop{n\equiv a\,( d)}}\Lambda(n)\big(\psi(-(n+1)^\gamma)-\psi(-n^\gamma)\big)e(t n^c)\,.
\end{equation}
Using \eqref{Gamma5} and Lemma \ref{Vaaler} we obtain
\begin{align}\label{Gamma5678}
\Gamma_5(x)=\Gamma_6(x)+\Gamma_7(x)+\Gamma_8(x)\,,
\end{align}
where
\begin{align}
\label{Gamma6}
&\Gamma_6(x)=\sum\limits_{x/2<n\leq x\atop{n\equiv a\,( d)}}\Lambda(n)
\sum\limits_{1\leq|h|\leq H}a(h)\Big(e(-h(n+1)^\gamma)-e(-hn^\gamma)\Big)e(tn^c)\,,\\
\label{Gamma7}
&\Gamma_7(x)\ll\sum\limits_{x/2<n\leq x\atop{n\equiv a\,( d)}}\Lambda(n)
\sum\limits_{|h|\leq H}b(h)e(-hn^\gamma)\,,\\
\label{Gamma8}
&\Gamma_8(x)\ll\sum\limits_{x/2<n\leq x\atop{n\equiv a\,( d)}}\Lambda(n)
\sum\limits_{|h|\leq H}b(h)e(-h(n+1)^\gamma)\,.
\end{align}
\textbf{Upper bound for} $\mathbf{\Gamma_6(x)}$

By \eqref{bh} and \eqref{Gamma6} we deduce
\begin{equation}\label{Gamma6est1}
\Gamma_6(x)\ll\sum\limits_{1\leq|h|\leq H}\frac{1}{|h|}
\Bigg|\sum\limits_{x/2<n\leq x\atop{n\equiv a\,( d)}}\Lambda(n)\Phi_h(n)e(tn^c-hn^\gamma)\Bigg|\,,
\end{equation}
where
\begin{equation*}
\Phi_h(y)=e\big(hy^\gamma-h(y+1)^\gamma\big)-1\,.
\end{equation*}
Bearing in mind the estimates
\begin{equation*}
\Phi_h(y)\ll|h|y^{\gamma-1}\,,  \quad \Phi'_h(y)\ll|h|y^{\gamma-2}
\end{equation*}
and using Abel's summation formula from \eqref{Gamma6est1} we derive
\begin{align}\label{Gamma6est2}
\Gamma_6(x)&\ll\sum\limits_{1\leq |h|\leq H}\frac{1}{|h|}\Bigg|\Phi_h(x)
\sum\limits_{x/2<n\leq x\atop{n\equiv a\,( d)}}\Lambda(n)e(tn^c+hn^\gamma)\Bigg|\nonumber\\
&+\sum\limits_{1\leq |h|\leq H}\frac{1}{|h|}\int\limits_{x/2}^{x}\Bigg|\Phi'_h(y)
\sum\limits_{x/2<n\leq y\atop{n\equiv a\,( d)}}\Lambda(n)e(tn^c+hn^\gamma)\Bigg|\,dy\nonumber\\
&\ll x^{\gamma-1}\big|\Gamma_9(x_1)\big|\,,
\end{align}
where
\begin{equation}\label{Gamma9}
\Gamma_9(x_1)=\sum\limits_{1\leq |h|\leq H}\Bigg|\sum\limits_{x/2<n\leq x_1\atop{n\equiv a\,( d)}}\Lambda(n)e(tn^c+hn^\gamma)\Bigg|
\end{equation}
for some number $x_1\in (x/2, x]$. Now the well-known formula
\begin{equation*}
\sum\limits_{k=1}^{d}e\left(\frac{kn}{d}\right)
=\begin{cases}d\,,\;\;\mbox{  if  }d\,|\,n,\\
0\,,\;\; \mbox{  if }d\nmid n
\end{cases}
\end{equation*}
and \eqref{Gamma9} lead to
\begin{equation}\label{Gamma9est1}
\Gamma_9(x_1)=\sum\limits_{1\leq |h|\leq H}\Bigg|\frac{1}{d}\sum\limits_{k=1}^{d}
\sum\limits_{x/2<n\leq x_1}\Lambda(n)e\left(tn^c+hn^\gamma+\frac{(n-a)k}{d}\right)\Bigg|\,.
\end{equation}
From \eqref{Gamma9est1} we understand that it is sufficient to estimate the sum
\begin{equation}\label{Gamma10}
\Gamma_{10}(x_1)=\sum\limits_{1\leq |h|\leq H}\Bigg|\sum\limits_{x/2<n\leq x_1}\Lambda(n)e\big(tn^c+hn^\gamma+kd^{-1}n\big)\Bigg|\,.
\end{equation}

\begin{lemma}\label{Gamma10est1}
Let $x_1\le x$, $0<\gamma<1<c$ and $|t|\leq x^{\gamma-c-\delta}$ for a sufficiently small $\delta>0$.
Then for the sum \eqref{Gamma10} we have
\begin{equation*}
\Gamma_{10}(x_1)\ll x^\frac{\delta}{2}\Big(H^\frac{7}{6} x^{\frac{\gamma}{6}+\frac{3}{4}}
+H^\frac{5}{4} x^{\frac{\gamma}{4}+\frac{5}{8}}+H^\frac{3}{4} x^{1-\frac{\gamma}{4}}+ H x^\frac{22}{25}\Big)\,.
\end{equation*}
\end{lemma}
\begin{proof}
It follows be the same arguments used in (\cite{Guo2}, pp. 10--11), since under the hypotheses on $t$, we have
\begin{equation*}
\left|\frac{d^j}{dy^j}\big( ty^c+hy^\gamma+kd^{-1}y\big)\right|
\asymp \left|\frac{d^j}{dy^j}\big(hy^\gamma+kd^{-1}y\big)\right|
\end{equation*}
for $j=2\,,3$ whenever $1\leq h \leq H$ and $x/2< y \leq x$.
\end{proof}

\begin{lemma}\label{SIest}
Set
\begin{equation}\label{SI}
S_I=\sum\limits_{1\leq |h|\leq H}\Bigg|\mathop{\sum\limits_{M<m\le M_1}a(m)
\sum\limits_{L<l\le L_1}}_{x/2<ml\le x_1}e\big(tm^cl^c+hm^\gamma l^\gamma+\theta ml\big)\Bigg|
\end{equation}
and
\begin{equation}\label{SI'}
S'_I=\sum\limits_{1\leq |h|\leq H}\Bigg|\mathop{\sum\limits_{M<m\le M_1}a(m)
\sum\limits_{L<l\le L_1}}_{x/2<ml\le x_1}e\big(tm^cl^c+hm^\gamma l^\gamma+\theta ml\big)\log l\Bigg|\,,
\end{equation}
where
\begin{equation}
\begin{split}\label{Conditions1}
&L\ge x^\frac{38c+115}{342}\,,\quad M_1\le 2M\,,\quad L_1\le 2L\,,\quad 0<\gamma<1<c<\frac{28}{19}\,,\\
&a(m)\ll \tau _5(m)\log x\,,\quad
0\leq\theta\leq1\,,\quad x_1\le x\,,\quad x^{\gamma-c-\delta}\leq|t|\leq x^\delta\,,
\end{split}
\end{equation}
with a sufficiently small $\delta>0$.
Then
\begin{equation*}
S_I,\, S'_I\ll  x^{\frac{\delta}{2}}\Big(Hx^{\frac{c}{9}+\frac{569}{684}}+H^\frac{7}{6}x^{\frac{\gamma}{6}-\frac{c}{18}+\frac{569}{684}}
+Hx^{1-\frac{\gamma}{3}}\Big)\,.
\end{equation*}
\end{lemma}
\begin{proof}
First we notice that \eqref{SI}  and \eqref{Conditions1} imply
\begin{equation}\label{LMasympX}
ML\asymp x\,.
\end{equation}
Denote
\begin{equation}\label{flm}
f(l)=tm^cl^c+hm^\gamma l^\gamma+\theta ml\,.
\end{equation}
By \eqref{SI}, \eqref{Conditions1} and \eqref{flm} we write
\begin{equation}\label{SIest1}
S_I\ll x^{\frac{\delta}{4}}\sum\limits_{1\leq |h|\leq H}\sum\limits_{M<m\le M_1}
\bigg|\sum\limits_{L'<l\leq L'_1}e\big(f(l)\big)\bigg|\,,
\end{equation}
where
\begin{equation}\label{L'L1'}
L'=\max{\bigg\{L,\frac{N}{m}\bigg\}}\,,\quad L_1'=\min{\bigg\{L_1,\frac{N_1}{m}\bigg\}}\,.
\end{equation}
From \eqref{Conditions1}, \eqref{flm} and \eqref{L'L1'} we have
\begin{equation}\label{'''est}
f'''(l)=c(c-1)(c-2)tm^cl^{c-3}+\gamma(\gamma-1)(\gamma-2)hm^\gamma l^{\gamma-3}\asymp |t|M^cL^{c-3}+|h|M^\gamma L^{\gamma-3}\,.
\end{equation}
Now \eqref{Conditions1}, \eqref{LMasympX}, \eqref{SIest1}, \eqref{'''est}  and Lemma \ref{Sargos}  yield
\begin{align*}
S_I&\ll x^{\frac{\delta}{4}}\sum\limits_{1\leq |h|\leq H}\sum\limits_{M<m\le M_1}
\Big(L\big(|t|M^cL^{c-3}+|h|M^\gamma L^{\gamma-3}\big)^\frac{1}{6}+\big(|t|M^cL^{c-3}+|h|M^\gamma L^{\gamma-3}\big)^{-\frac{1}{3}}\Big)\\
&\ll x^{\frac{\delta}{4}}\sum\limits_{1\leq |h|\leq H}\sum\limits_{M<m\le M_1}
\Big(|t|^\frac{1}{6}M^\frac{c}{6}L^{\frac{c}{6}+\frac{1}{2}}+|h|^\frac{1}{6}M^\frac{\gamma}{6}L^{\frac{\gamma}{6}+\frac{1}{2}}
+|t|^{-\frac{1}{3}}M^{-\frac{c}{3}}L^{1-\frac{c}{3}}+|h|^{-\frac{1}{3}}M^{-\frac{\gamma}{3}}L^{1-\frac{\gamma}{3}}\Big)\\
&\ll x^{\frac{\delta}{4}}\sum\limits_{1\leq |h|\leq H}
\Big(|t|^\frac{1}{6}x^{\frac{c}{6}+\frac{1}{2}}M^\frac{1}{2}+|h|^\frac{1}{6}x^{\frac{\gamma}{6}+\frac{1}{2}}M^\frac{1}{2}
+|t|^{-\frac{1}{3}}x^{1-\frac{c}{3}}+|h|^{-\frac{1}{3}}x^{1-\frac{\gamma}{3}}\Big)\\                                                                                                                 &\ll x^{\frac{\delta}{4}}\Big(H|t|^\frac{1}{6}x^{\frac{c}{6}+\frac{1}{2}}M^\frac{1}{2}+H^\frac{7}{6}x^{\frac{\gamma}{6}+\frac{1}{2}}M^\frac{1}{2}
+H|t|^{-\frac{1}{3}}x^{1-\frac{c}{3}}+H^\frac{2}{3}x^{1-\frac{\gamma}{3}}\Big)\\
&\ll x^{\frac{\delta}{4}}\Big(Hx^{\frac{c}{6}+\frac{1}{2}+\frac{\delta}{6}}M^\frac{1}{2}
+H^\frac{7}{6}x^{\frac{\gamma}{6}+\frac{1}{2}}M^\frac{1}{2}+Hx^{1-\frac{\gamma}{3}+\frac{\delta}{3}}+H^\frac{2}{3}x^{1-\frac{\gamma}{3}}\Big)\\
&\ll x^{\frac{\delta}{2}}\Big(Hx^{\frac{c}{9}+\frac{569}{684}}+H^\frac{7}{6}x^{\frac{\gamma}{6}-\frac{c}{18}+\frac{569}{684}}
+Hx^{1-\frac{\gamma}{3}}\Big)\,.
\end{align*}
To estimate the sum defined by \eqref{SI'} we apply Abel's summation formula and proceed in the same way to deduce
\begin{equation*}
S_I'\ll  x^{\frac{\delta}{2}}\Big(Hx^{\frac{c}{9}+\frac{569}{684}}
+H^\frac{7}{6}x^{\frac{\gamma}{6}-\frac{c}{18}+\frac{569}{684}}+Hx^{1-\frac{\gamma}{3}}\Big)\,.
\end{equation*}
This proves the lemma.
\end{proof}

\begin{lemma}\label{SIIest}
Set
\begin{equation}\label{SII}
S_{II}=\sum\limits_{1\leq |h|\leq H}\Bigg|\mathop{\sum\limits_{M<m\le M_1}a(m)
\sum\limits_{L<l\le L_1}}_{x/2<ml\le x_1}b(l)e\big(tm^cl^c+hm^\gamma l^\gamma+\theta ml\big)\Bigg|\,,
\end{equation}
where
\begin{equation}
\begin{split}\label{Conditions2}
&2^{-10}x^\frac{56-38c}{171}\leq L\leq 2^7x^\frac{1}{3}\,,\quad M_1\le 2M\,,\quad L_1\le 2L\,,\quad 0<\gamma<1<c<\frac{28}{19}\,,\\
&a(m)\ll \tau _5(m)\log x\,,\;\;b(l)\ll \tau _5(l)\log x\,,\;\; 0\leq\theta\leq1\,,\;\;x_1\le x\,,\;\; x^{\gamma-c-\delta}\leq|t|\leq x^\delta\,,
\end{split}
\end{equation}
with a sufficiently small $\delta>0$.
Then
\begin{align*}
S_{II}&\ll  x^\frac{\delta}{2}\Big(Hx^{\frac{c}{4}+\frac{7}{12}}
+H^\frac{5}{4}x^{\frac{\gamma}{4}+\frac{7}{12}}+Hx^{\frac{c}{9}+\frac{143}{171}}+Hx^{1-\frac{\gamma}{4}}
+Hx^{\frac{c}{6}+\frac{13}{18}}+H^\frac{7}{6}x^{\frac{\gamma}{6}+\frac{13}{18}}+H^\frac{9}{8}x^\frac{5}{6}\\
&\hspace{118mm}+H^\frac{7}{8}x^{\frac{c}{8}-\frac{\gamma}{8}+\frac{5}{6}}\Big)\,.
\end{align*}
\end{lemma}
\begin{proof}
First we notice that  \eqref{SII}  and \eqref{Conditions2} give us
\begin{equation}\label{LMasympX2}
ML\asymp x\,.
\end{equation}
From  \eqref{flm}, \eqref{Conditions2}, \eqref{LMasympX2}, Cauchy's inequality
and Lemma \ref{Squareoutlemma} with
\begin{equation}\label{QL}
1\leq Q\leq L
\end{equation}
it follows
\begin{align}\label{SIIest1}
&\Bigg|\mathop{\sum\limits_{M<m\le M_1}a(m)
\sum\limits_{L<l\le L_1}}_{x/2<ml\le x_1}b(l)e\big(tm^cl^c+hm^\gamma l^\gamma+\theta ml\big)\Bigg|^2\nonumber\\
&\ll\sum\limits_{M<m\le M_1}|a(m)|^2\sum\limits_{M<m\le M_1}\bigg|\sum\limits_{L<l\le L_1\atop{x/2<ml\le x_1}}
b(l)e\big(f(l)\big)\bigg|^2\nonumber\\
&\ll M^{1+\frac{\delta}{2}}\sum\limits_{M<m\le M_1}\frac{L}{Q}\sum_{|q|<Q}\bigg(1-\frac{q}{Q}\bigg)
\sum\limits_{L<l, \, l+q\leq L_1\atop{x/2<ml\le x_1\atop{x<m(l+q)\le x_1}}}b(l+q)\overline{b(l)}
e\big(f(l+q)-f(l)\big)\nonumber\\
&\ll \frac{LM^{1+\frac{\delta}{2}}}{Q}\sum\limits_{M<m\le M_1}\Bigg(L^{1+\frac{\delta}{2}}+\sum_{1\leq |q|<Q}\bigg(1-\frac{q}{Q}\bigg)
\sum\limits_{L<l, \, l+q\leq L_1\atop{x<ml\le x_1\atop{x<m(l+q)\le x_1}}}
b(l+q)\overline{b(l)}e\big(f(l+q)-f(l)\big)\Bigg)\nonumber\\
&\ll x^\frac{\delta}{2}\Bigg(\frac{x^2}{Q}+\frac{x}{Q}\sum\limits_{1\leq |q|\leq Q}
\sum\limits_{L<l, \, l+q\leq L_1}\bigg|\sum\limits_{M'<m\leq M_1'}e\big(g(m)\big)\bigg|\Bigg)\,,
\end{align}
where
\begin{equation}\label{M'M1'}
M'=\max{\bigg\{M,\frac{x}{l},\frac{x}{l+q}\bigg\}}\,,
\quad M_1'=\min{\bigg\{M_1,\frac{x_1}{l},\frac{x_1}{l+q}\bigg\}}
\end{equation}
and
\begin{equation}\label{gm}
g(m)=tm^c(l+q)^c-tm^cl^c+hm^\gamma (l+q)^\gamma-hm^\gamma l^\gamma+\theta mq\,.
\end{equation}
By \eqref{Conditions2}, \eqref{M'M1'} and \eqref{gm} we have
\begin{equation}\label{''est}
g''(m)\asymp |t|M^{c-2}L^{c-1}|q|+|h|M^{\gamma-2} L^{\gamma-1}|q|\,.
\end{equation}
Now \eqref{Conditions2}, \eqref{LMasympX2}, \eqref{SIIest1}, \eqref{''est} and Lemma \ref{Korput} lead to
\begin{align}\label{SIIest2}
&\Bigg|\mathop{\sum\limits_{M<m\le M_1}a(m)
\sum\limits_{L<l\le L_1}}_{x/2<ml\le x_1}b(l)e\big(tm^cl^c+hm^\gamma l^\gamma+\theta ml\big)\Bigg|^2\nonumber\\
&\ll x^\frac{\delta}{2}\Bigg(\frac{x^2}{Q}+\frac{x}{Q}\sum\limits_{1\leq |q|\leq Q}
\sum\limits_{L<l\leq L_1}M\Big(|t|M^{c-2}L^{c-1}|q|+|h|M^{\gamma-2} L^{\gamma-1}|q|\Big)^\frac{1}{2} \nonumber\\
&\hspace{18mm}+\Big(|t|M^{c-2}L^{c-1}|q|+|h|M^{\gamma-2} L^{\gamma-1}|q|\Big)^{-\frac{1}{2}}\Bigg)\nonumber\\
&\ll x^\frac{\delta}{2}\Bigg(\frac{x^2}{Q}+\frac{x}{Q}\sum\limits_{1\leq |q|\leq Q}
\sum\limits_{L<l\leq L_1}|t|^\frac{1}{2}M^\frac{c}{2}L^{\frac{c}{2}-\frac{1}{2}}|q|^\frac{1}{2}
+|h|^\frac{1}{2}M^\frac{\gamma}{2} L^{\frac{\gamma}{2}-\frac{1}{2}}|q|^\frac{1}{2} \nonumber\\
&\hspace{18mm}+|t|^{-\frac{1}{2}}M^{1-\frac{c}{2}}L^{\frac{1}{2}-\frac{c}{2}}|q|^{-\frac{1}{2}}
+|h|^{-\frac{1}{2}}M^{1-\frac{\gamma}{2}}L^{\frac{1}{2}-\frac{\gamma}{2}}|q|^{-\frac{1}{2}}\Bigg)\nonumber\\
&\ll x^\frac{\delta}{2}\Big(x^2Q^{-1}+x^{1+\frac{c}{2}}|t|^\frac{1}{2}L^\frac{1}{2}Q^\frac{1}{2}
+x^{1+\frac{\gamma}{2}}|h|^\frac{1}{2}L^\frac{1}{2}Q^\frac{1}{2} \nonumber\\
&\hspace{18mm}+x^{2-\frac{c}{2}}|t|^{-\frac{1}{2}}L^\frac{1}{2}Q^{-\frac{1}{2}}
+x^{2-\frac{\gamma}{2}}|h|^{-\frac{1}{2}}L^\frac{1}{2}Q^{-\frac{1}{2}}\Big)\,.
\end{align}
From \eqref{QL}, \eqref{SIIest2} and Lemma \ref{Optimization} it follows that there exists an optimal $Q$ such that
\begin{align*}
&\Bigg|\mathop{\sum\limits_{M<m\le M_1}a(m)
\sum\limits_{L<l\le L_1}}_{x/2<ml\le x_1}b(l)e\big(tm^cl^c+hm^\gamma l^\gamma+\theta ml\big)\Bigg|^2\\
&\ll x^\frac{\delta}{2}\Big(x^{1+\frac{c}{2}}|t|^\frac{1}{2}L^\frac{1}{2}
+x^{1+\frac{\gamma}{2}}|h|^\frac{1}{2}L^\frac{1}{2}+x^2L^{-1}
+x^{2-\frac{c}{2}}|t|^{-\frac{1}{2}}+x^{2-\frac{\gamma}{2}}|h|^{-\frac{1}{2}}\\
&+x^{\frac{4}{3}+\frac{c}{3}}|t|^\frac{1}{3}L^\frac{1}{3}+x^{\frac{4}{3}+\frac{\gamma}{3}}|h|^\frac{1}{3}L^\frac{1}{3}
+x^\frac{3}{2}L^\frac{1}{2}+x^{\frac{3}{2}-\frac{c}{4}+\frac{\gamma}{4}}|t|^{-\frac{1}{4}}|h|^\frac{1}{4}L^\frac{1}{2}
+x^{\frac{3}{2}+\frac{c}{4}-\frac{\gamma}{4}}|t|^\frac{1}{4}|h|^{-\frac{1}{4}}L^\frac{1}{2}\Big)\,.\\
\end{align*}
Therefore
\begin{align}\label{SIIest3}
&\Bigg|\mathop{\sum\limits_{M<m\le M_1}a(m)
\sum\limits_{L<l\le L_1}}_{x/2<ml\le x_1}b(l)e\big(tm^cl^c+hm^\gamma l^\gamma+\theta ml\big)\Bigg|\nonumber\\
&\ll x^\frac{\delta}{4}\Big(x^{\frac{1}{2}+\frac{c}{4}}|t|^\frac{1}{4}L^\frac{1}{4}
+x^{\frac{1}{2}+\frac{\gamma}{4}}|h|^\frac{1}{4}L^\frac{1}{4}+xL^{-\frac{1}{2}}
+x^{1-\frac{c}{4}}|t|^{-\frac{1}{4}}+x^{1-\frac{\gamma}{4}}|h|^{-\frac{1}{4}}
+x^{\frac{2}{3}+\frac{c}{6}}|t|^\frac{1}{6}L^\frac{1}{6}\nonumber\\
&+x^{\frac{2}{3}+\frac{\gamma}{6}}|h|^\frac{1}{6}L^\frac{1}{6}
+x^\frac{3}{4}L^\frac{1}{4}+x^{\frac{3}{4}-\frac{c}{8}+\frac{\gamma}{8}}|t|^{-\frac{1}{8}}|h|^\frac{1}{8}L^\frac{1}{4}
+x^{\frac{3}{4}+\frac{c}{8}-\frac{\gamma}{8}}|t|^\frac{1}{8}|h|^{-\frac{1}{8}}L^\frac{1}{4}\Big)\,.
\end{align}
Bearing in mind \eqref{SII}, \eqref{Conditions2} and \eqref{SIIest3} we derive
\begin{align*}
S_{II}&\ll x^\frac{\delta}{4}\sum\limits_{1\leq |h|\leq H}\Big(x^{\frac{1}{2}+\frac{c}{4}}|t|^\frac{1}{4}L^\frac{1}{4}
+x^{\frac{1}{2}+\frac{\gamma}{4}}|h|^\frac{1}{4}L^\frac{1}{4}+xL^{-\frac{1}{2}}
+x^{1-\frac{c}{4}}|t|^{-\frac{1}{4}}+x^{1-\frac{\gamma}{4}}|h|^{-\frac{1}{4}}\\
&+x^{\frac{2}{3}+\frac{c}{6}}|t|^\frac{1}{6}L^\frac{1}{6}+x^{\frac{2}{3}+\frac{\gamma}{6}}|h|^\frac{1}{6}L^\frac{1}{6}
+x^\frac{3}{4}L^\frac{1}{4}+x^{\frac{3}{4}-\frac{c}{8}+\frac{\gamma}{8}}|t|^{-\frac{1}{8}}|h|^\frac{1}{8}L^\frac{1}{4}
+x^{\frac{3}{4}+\frac{c}{8}-\frac{\gamma}{8}}|t|^\frac{1}{8}|h|^{-\frac{1}{8}}L^\frac{1}{4}\Big)\\
&\ll x^\frac{\delta}{4}\Big(Hx^{\frac{1}{2}+\frac{c}{4}}|t|^\frac{1}{4}L^\frac{1}{4}
+H^\frac{5}{4}x^{\frac{1}{2}+\frac{\gamma}{4}}L^\frac{1}{4}+HxL^{-\frac{1}{2}}
+Hx^{1-\frac{c}{4}}|t|^{-\frac{1}{4}}+H^\frac{3}{4}x^{1-\frac{\gamma}{4}}\\
&+Hx^{\frac{2}{3}+\frac{c}{6}}|t|^\frac{1}{6}L^\frac{1}{6}+H^\frac{7}{6}x^{\frac{2}{3}+\frac{\gamma}{6}}L^\frac{1}{6}
+Hx^\frac{3}{4}L^\frac{1}{4}+H^\frac{9}{8}x^{\frac{3}{4}-\frac{c}{8}+\frac{\gamma}{8}}|t|^{-\frac{1}{8}}L^\frac{1}{4}
+H^\frac{7}{8}x^{\frac{3}{4}+\frac{c}{8}-\frac{\gamma}{8}}|t|^\frac{1}{8}L^\frac{1}{4}\Big)\\
&\ll x^\frac{\delta}{4}\Big(Hx^{\frac{1}{2}+\frac{c}{4}+\frac{\delta}{4}}L^\frac{1}{4}
+H^\frac{5}{4}x^{\frac{1}{2}+\frac{\gamma}{4}}L^\frac{1}{4}+HxL^{-\frac{1}{2}}
+Hx^{1-\frac{\gamma}{4}+\frac{\delta}{4}}+H^\frac{3}{4}x^{1-\frac{\gamma}{4}}
+Hx^{\frac{2}{3}+\frac{c}{6}+\frac{\delta}{6}}L^\frac{1}{6}\\
&\hspace{44mm}+H^\frac{7}{6}x^{\frac{2}{3}+\frac{\gamma}{6}}L^\frac{1}{6}
+Hx^\frac{3}{4}L^\frac{1}{4}+H^\frac{9}{8}x^{\frac{3}{4}+\frac{\delta}{8}}L^\frac{1}{4}
+H^\frac{7}{8}x^{\frac{3}{4}+\frac{c}{8}-\frac{\gamma}{8}+\frac{\delta}{8}}L^\frac{1}{4}\Big)\\
&\ll x^\frac{\delta}{2}\Big(Hx^{\frac{1}{2}+\frac{c}{4}}L^\frac{1}{4}+
H^\frac{5}{4}x^{\frac{1}{2}+\frac{\gamma}{4}}L^\frac{1}{4}+HxL^{-\frac{1}{2}}+Hx^{1-\frac{\gamma}{4}}
+Hx^{\frac{2}{3}+\frac{c}{6}}L^\frac{1}{6}+H^\frac{7}{6}x^{\frac{2}{3}+\frac{\gamma}{6}}L^\frac{1}{6}\\
&\hspace{93.5mm}+H^\frac{9}{8}x^\frac{3}{4}L^\frac{1}{4}
+H^\frac{7}{8}x^{\frac{3}{4}+\frac{c}{8}-\frac{\gamma}{8}}L^\frac{1}{4}\Big)\\
&\ll x^\frac{\delta}{2}\Big(Hx^{\frac{c}{4}+\frac{7}{12}}
+H^\frac{5}{4}x^{\frac{\gamma}{4}+\frac{7}{12}}+Hx^{\frac{c}{9}+\frac{143}{171}}+Hx^{1-\frac{\gamma}{4}}
+Hx^{\frac{c}{6}+\frac{13}{18}}+H^\frac{7}{6}x^{\frac{\gamma}{6}+\frac{13}{18}}+H^\frac{9}{8}x^\frac{5}{6}\\
&\hspace{118mm}+H^\frac{7}{8}x^{\frac{c}{8}-\frac{\gamma}{8}+\frac{5}{6}}\Big)\,.
\end{align*}
This proves the lemma.
\end{proof}
It remains to apply Heath-Brown's identity to the sum  $\Gamma_{10}(x_1)$.
\begin{lemma}\label{Gamma10est2}
Let $x_1\le x$, $0<\gamma<1<c<\frac{28}{19}$ and  $x^{\gamma-c-\delta}\leq|t|\leq x^\delta$ with a sufficiently small $\delta>0$.
Then for the sum \eqref{Gamma10} we have
\begin{align*}
\Gamma_{10}(x_1)&\ll  x^\frac{\delta}{2}\Big(Hx^{\frac{c}{4}+\frac{7}{12}}
+H^\frac{5}{4}x^{\frac{\gamma}{4}+\frac{7}{12}}+Hx^{\frac{c}{9}+\frac{143}{171}}
+Hx^{1-\frac{\gamma}{4}}+Hx^{\frac{c}{6}+\frac{13}{18}}
+H^\frac{7}{6}x^{\frac{\gamma}{6}+\frac{13}{18}}\\
&\hspace{60mm}+H^\frac{9}{8}x^\frac{5}{6}+H^\frac{7}{6}x^{\frac{\gamma}{6}-\frac{c}{18}+\frac{569}{684}}
+H^\frac{7}{8}x^{\frac{c}{8}-\frac{\gamma}{8}+\frac{5}{6}}\Big)\,.
\end{align*}
\end{lemma}
\begin{proof}
Take
\begin{equation}\label{UVZ}
U=2^{-10}x^\frac{56-38c}{171}\,,\quad V=2^7x^\frac{1}{3}\,,\quad Z=x^\frac{38c+115}{342}\,.
\end{equation}
According to Lemma \ref{Heath-Brown}, the sum $\Gamma_{10}(x_1)$
can be decomposed into $O\Big(\log^6x\Big)$ sums, each of which is either of Type I
\begin{equation*}
S_I=\sum\limits_{1\leq |h|\leq H}\Bigg|\mathop{\sum\limits_{M<m\le M_1}a(m)
\sum\limits_{L<l\le L_1}}_{x/2<ml\le x_1}e\big(tm^cl^c+hm^\gamma l^\gamma+kd^{-1} ml\big)\Bigg|
\end{equation*}
and
\begin{equation*}
S'_I=\sum\limits_{1\leq |h|\leq H}\Bigg|\mathop{\sum\limits_{M<m\le M_1}a(m)
\sum\limits_{L<l\le L_1}}_{x/2<ml\le x_1}e\big(tm^cl^c+hm^\gamma l^\gamma+kd^{-1}ml\big)\log l\Bigg|\,,
\end{equation*}
where
\begin{equation*}
L\ge Z\,,\quad M_1\le 2M\,,\;\;  L_1\le 2L\,,\quad  a(m)\ll \tau _5(m)\log x
\end{equation*}
or of Type II
\begin{equation*}
S_{II}=\sum\limits_{1\leq |h|\leq H}\Bigg|\mathop{\sum\limits_{M<m\le M_1}a(m)
\sum\limits_{L<l\le L_1}}_{x/2<ml\le x_1}b(l)e\big(tm^cl^c+hm^\gamma l^\gamma+kd^{-1} ml\big)\Bigg|\,,
\end{equation*}
where
\begin{equation*}
U\le L\le V\,,\quad M_1\le 2M\,,\quad L_1\le 2L\,,\quad
a(m)\ll \tau _5(m)\log x\,,\quad b(l)\ll \tau _5(l)\log x\,.
\end{equation*}
Bearing in mind  \eqref{UVZ}, Lemma \ref{SIest} and  Lemma \ref{SIIest} we establish the statement in the lemma.
\end{proof}
\begin{lemma}\label{Gamma10est3}
Let $x_1\le x$, $0<\gamma<1<c<\frac{28}{19}$ and $|t|\leq x^\delta$ with a sufficiently small $\delta>0$.
Then for the sum \eqref{Gamma10} we have
\begin{align*}
\Gamma_{10}(x_1)&\ll  x^\frac{\delta}{2}\Big(Hx^{\frac{c}{4}+\frac{7}{12}}
+H^\frac{5}{4}x^{\frac{\gamma}{4}+\frac{5}{8}}+Hx^{\frac{c}{9}+\frac{143}{171}}
+Hx^{1-\frac{\gamma}{4}}+Hx^{\frac{c}{6}+\frac{13}{18}}+H^\frac{7}{6}x^{\frac{\gamma}{6}+\frac{3}{4}}\\
&\hspace{60mm}+H^\frac{9}{8}x^\frac{5}{6}+H^\frac{7}{6}x^{\frac{\gamma}{6}-\frac{c}{18}+\frac{569}{684}}
+H^\frac{7}{8}x^{\frac{c}{8}-\frac{\gamma}{8}+\frac{5}{6}}\Big)\,.
\end{align*}
\end{lemma}
\begin{proof}
It follows immediately from Lemma \ref{Gamma10est1} and Lemma \ref{Gamma10est2}.
\end{proof}
Now \eqref{Gamma6est2}, \eqref{Gamma9est1}, \eqref{Gamma10} and Lemma \ref{Gamma10est3} imply
\begin{align}\label{Gamma6est3}
\Gamma_6(x)&\ll  x^\frac{\delta}{2}\Big(Hx^{\frac{c}{4}+\gamma-\frac{5}{12}}
+H^\frac{5}{4}x^{\frac{5\gamma}{4}-\frac{3}{8}}+Hx^{\frac{c}{9}+\gamma-\frac{28}{171}}
+Hx^\frac{3\gamma}{4}+Hx^{\frac{c}{6}+\gamma-\frac{5}{18}}+H^\frac{7}{6}x^{\frac{7\gamma}{6}-\frac{1}{4}}\nonumber\\
&\hspace{54mm}+H^\frac{9}{8}x^{\gamma-\frac{1}{6}}+H^\frac{7}{6}x^{\frac{7\gamma}{6}-\frac{c}{18}-\frac{115}{684}}
+H^\frac{7}{8}x^{\frac{c}{8}+\frac{7\gamma}{8}-\frac{1}{6}}\Big)\,.
\end{align}

\vspace{20mm}

\textbf{Upper bound for} $\mathbf{\Gamma_7(x)}$

By \eqref{Siegel-Walfisz}, \eqref{bh} and \eqref{Gamma7} we deduce
\begin{align}\label{Gamma7est1}
\Gamma_7(x)&\ll |b_0|\sum\limits_{x/2<n\leq x\atop{n\equiv a\,( d)}}\Lambda(n)
+\sum\limits_{1\leq|h|\leq H}|b(h)|\Bigg|\sum\limits_{x/2<n\leq x\atop{n\equiv a\,( d)}}\Lambda(n)e(-hn^\gamma)\Bigg|\nonumber\\
&\ll\frac{x}{H\varphi(d)}
+\frac{1}{H}\sum\limits_{1\leq|h|\leq H}\Bigg|\sum\limits_{x/2<n\leq x\atop{n\equiv a\,( d)}}\Lambda(n)e(-hn^\gamma)\Bigg|\nonumber\\
&\ll H^{-1}x+H^{-1}\Gamma_{11}(x)\,,
\end{align}
where
\begin{equation*}
\Gamma_{11}(x)=\sum\limits_{1\leq|h|\leq H}\Bigg|\sum\limits_{x/2<n\leq x\atop{n\equiv a\,( d)}}\Lambda(n)e(-hn^\gamma)\Bigg|\,.
\end{equation*}
Estimating the sum $\Gamma_{11}$(x) as in (\cite{Guo2}, (3.7)) we obtain
\begin{equation}\label{Gamma11est1}
\Gamma_{11}(x)\ll x^\frac{\delta}{2}\Big(H^\frac{7}{6} x^{\frac{\gamma}{6}+\frac{3}{4}}
+H^\frac{5}{4} x^{\frac{\gamma}{4}+\frac{5}{8}}+H^\frac{3}{4} x^{1-\frac{\gamma}{4}}+ H x^\frac{22}{25}\Big)\,.
\end{equation}
Now  \eqref{Gamma7est1} and \eqref{Gamma11est1} give us
\begin{equation}\label{Gamma7est2}
\Gamma_7(x)\ll x^\frac{\delta}{2}\Big(H^{-1}x+H^\frac{1}{6} x^{\frac{\gamma}{6}+\frac{3}{4}}
+H^\frac{1}{4} x^{\frac{\gamma}{4}+\frac{5}{8}}+H^{-\frac{1}{4}} x^{1-\frac{\gamma}{4}}+x^\frac{22}{25}\Big)\,.
\end{equation}
\textbf{Upper bound for} $\mathbf{\Gamma_8(x)}$

In the same way for the sum defined by \eqref{Gamma8} we get
\begin{equation}\label{Gamma8est1}
\Gamma_8(x)\ll x^\frac{\delta}{2}\Big(H^{-1}x+H^\frac{1}{6} x^{\frac{\gamma}{6}+\frac{3}{4}}
+H^\frac{1}{4} x^{\frac{\gamma}{4}+\frac{5}{8}}+H^{-\frac{1}{4}} x^{1-\frac{\gamma}{4}}+x^\frac{22}{25}\Big)\,.
\end{equation}
From \eqref{Gamma5678}, \eqref{Gamma6est3}, \eqref{Gamma7est2} and \eqref{Gamma8est1} we derive
\begin{align}\label{Gamma5est1}
\Gamma_5(x)&\ll x^\frac{\delta}{2}\Big(Hx^{\frac{c}{4}+\gamma-\frac{5}{12}}
+H^\frac{5}{4}x^{\frac{5\gamma}{4}-\frac{3}{8}}+Hx^{\frac{c}{9}+\gamma-\frac{28}{171}}
+Hx^\frac{3\gamma}{4}+Hx^{\frac{c}{6}+\gamma-\frac{5}{18}}+H^\frac{7}{6}x^{\frac{7\gamma}{6}-\frac{1}{4}}\nonumber\\
&+H^\frac{9}{8}x^{\gamma-\frac{1}{6}}+H^\frac{7}{6}x^{\frac{7\gamma}{6}-\frac{c}{18}-\frac{115}{684}}
+H^\frac{7}{8}x^{\frac{c}{8}+\frac{7\gamma}{8}-\frac{1}{6}}+H^\frac{1}{6} x^{\frac{\gamma}{6}+\frac{3}{4}}
+H^\frac{1}{4} x^{\frac{\gamma}{4}+\frac{5}{8}}+H^{-\frac{1}{4}} x^{1-\frac{\gamma}{4}}\nonumber\\
&\hspace{105mm}+H^{-1}x+x^\frac{22}{25}\Big)\,.
\end{align}
Taking into account that \eqref{Gamma5est1} holds for any real $H\geq1$ and using Lemma \ref{Optimization} we write
\begin{align}\label{Gamma5est2}
\Gamma_5(x)&\ll x^\frac{\delta}{2}\Big(x^{\frac{c}{4}+\gamma-\frac{5}{12}}
+x^{\frac{5\gamma}{4}-\frac{3}{8}}+x^{\frac{c}{9}+\gamma-\frac{28}{171}}
+x^\frac{3\gamma}{4}+x^{\frac{c}{6}+\gamma-\frac{5}{18}}+x^{\frac{7\gamma}{6}-\frac{1}{4}}+x^{\gamma-\frac{1}{6}}\nonumber\\
&+x^{\frac{7\gamma}{6}-\frac{c}{18}-\frac{115}{684}}
+x^{\frac{c}{8}+\frac{7\gamma}{8}-\frac{1}{6}}+x^{\frac{\gamma}{6}+\frac{3}{4}}
+x^{\frac{\gamma}{4}+\frac{5}{8}}+x^\frac{22}{25}+x^{\frac{c}{20}+\frac{43}{60}}+x^\frac{37}{48}+x^{\frac{c}{45}+\frac{656}{855}}\nonumber\\
&+x^{\frac{4}{5}-\frac{\gamma}{20}}+x^{\frac{c}{30}+\frac{67}{90}}+x^\frac{53}{68}+x^{\frac{26}{33}-\frac{\gamma}{44}}
+x^{\frac{20768}{26163}-\frac{c}{102}}+x^{\frac{c}{36}+\frac{20}{27}}+x^\frac{17}{20}+x^\frac{13}{16}
+x^{\frac{c}{8}+\frac{\gamma}{2}+\frac{7}{24}}\nonumber\\
&+x^{\frac{5\gamma}{9}+\frac{7}{18}}+x^{\frac{c}{18}+\frac{\gamma}{2}+\frac{143}{342}}+x^{\frac{3\gamma}{8}+\frac{1}{2}}
+x^{\frac{c}{12}+\frac{\gamma}{2}+\frac{13}{36}}+x^{\frac{7\gamma}{13}+\frac{11}{26}}+x^{\frac{8\gamma}{17}+\frac{23}{51}}
+x^{\frac{7\gamma}{13}-\frac{c}{39}+\frac{683}{1482}}\nonumber\\
&+x^{\frac{c}{15}+\frac{7\gamma}{15}+\frac{17}{45}}+x^{\frac{\gamma}{7}+\frac{11}{14}}+x^{\frac{\gamma}{5}+\frac{7}{10}}\,.
\end{align}
Using \eqref{cgamma} and \eqref{Gamma5est2} we find
\begin{equation}\label{Gamma5est3}
\Gamma_5(x)\ll x^{\frac{c}{18}+\frac{\gamma}{2}+\frac{143}{342}+\frac{\delta}{2}}\,.
\end{equation}
Bearing in mind  \eqref{Gamma23}, \eqref{Gamma34}, \eqref{Gamma45} and \eqref{Gamma5est3} we deduce
\begin{equation}\label{Gamma2est1}
\Gamma_2(x)\ll x^{\frac{c}{18}+\frac{\gamma}{2}+\frac{143}{342}+\delta}\,.
\end{equation}
It is easy to see that \eqref{cgamma} and \eqref{Gamma2est1} yield
\begin{equation*}
 \Gamma_2(x)\ll x^{\gamma-\varepsilon}\,,
\end{equation*}
for some $\varepsilon>0$.

\subsection{The end of the proof}\label{Sectionfinal}
\indent

Summarizing \eqref{Gamma12}, \eqref{Gamma1est1} and \eqref{Gamma2est1}
we establish the asymptotic formula \eqref{asymptoticformula1}.

This completes the proof of  Theorem \ref{Theorem}.

\vspace{5mm}

\textbf{Acknowledgements.} Professor Victor Guo kindly sent his paper \cite{Guo2} to the author
and the author is very grateful to him.

\vskip20pt
\footnotesize
\begin{flushleft}
S. I. Dimitrov\\
Faculty of Applied Mathematics and Informatics\\
Technical University of Sofia \\
Blvd. St.Kliment Ohridski 8, \\
Sofia 1756, Bulgaria\\
e-mail: sdimitrov@tu-sofia.bg\\
\end{flushleft}

\end{document}